\newif\ifpreprint
\preprintfalse
\newcommand{\theTitle}{Learning Rationality in Potential Games}
\newcommand{\theAbstract}{We propose a stochastic first-order algorithm to learn the \emph{rationality} parameters of simultaneous and non-cooperative potential games, \ie, the parameters of the agents' optimization problems. Our technique combines an active-set step that enforces that the agents play at a Nash equilibrium and an implicit-differentiation step to update the estimates of the rationality parameters. We detail the convergence properties of our algorithm and perform numerical experiments on Cournot and congestion games. In practice, we show that our algorithm effectively finds high-quality solutions with minimal out-of-sample loss and scales to large datasets.}

\ifpreprint
    \documentclass[12pt]{article}
    \usepackage{graphbox}
    \usepackage[margin=1in]{geometry}
    \usepackage[sort&compress,numbers]{natbib}
    \usepackage[english]{babel}
    \usepackage{changepage}
    \usepackage[utf8]{inputenc}

\else
    \documentclass[letterpaper, 10pt, conference]{ieeeconf}
    \makeatletter
    \let\NAT@parse\undefined
    \makeatother
    \let\labelindent\relax

    \usepackage{balance}
    \usepackage[strict]{changepage}
     \usepackage[numbers]{natbib}
     
\fi

\usepackage{url}
\usepackage[colorlinks]{hyperref}
\usepackage{graphicx}
\usepackage{pgfplots}
\usepackage[english]{babel}
\usepackage{blindtext}
\usepackage{amsthm, mathtools, amssymb}
\usepackage{multirow}
\usepackage{comment}
\usepackage[inline]{enumitem}
\usepackage[overload]{empheq}
\usepackage{algorithm,algorithmic}

\usepackage{csvsimple}	
\usepackage{booktabs}   
\usepackage{adjustbox}  
\pgfplotsset{compat=1.18}

\newcommand{\new}[1]{\textcolor{black}{#1}}

\newcommand{\eg}{{\it e.g.}}
\newcommand{\ie}{{\it i.e.}}

\newcommand{\nagent}{n}

\newcommand{\param}{\theta}

\newcommand{\gameparam}{\mu}

\newcommand{\mini}{\text{\rm minimize}}

\newcommand{\reals}{{\mbox{\bf R}}}

\newcommand{\Expect}{\mathop{\bf E{}}}

\newcommand{\loss}{\mathcal{L}}

\newtheorem{theorem}{Theorem}[section]  
\newtheorem{proposition}{Proposition}[section]
\newtheorem{assumption}{Assumption}
\newtheorem{definition}{Definition}[section]

\newtheorem{zrule}{Rule}

\usepackage{cleveref}
\crefname{lemma}{Lemma}{Lemmata}
\crefname{theorem}{Theorem}{Theorem}
\crefname{claim}{Claim}{Claims}
\crefname{proposition}{Proposition}{Propositions}
\crefname{algorithm}{Algorithm}{Algorithms}
\crefname{equation}{}{}
\crefname{definition}{Definition}{Definitions}
\crefname{Cla}{Claim}{Claim}
\crefname{corollary}{Corollary}{Corollaries}
\crefname{remark}{Remark}{Remarks}
\crefname{example}{Example}{Examples}
\crefname{figure}{Figure}{Figures}
\crefname{section}{Section}{Sections}
\crefname{table}{Table}{Tables}
\crefname{enumi}{Statement}{Statements}
\crefname{line}{Step}{Steps}
\crefname{zrule}{Rule}{Rules}
\crefname{assumption}{Assumption}{Assumptions}
\newcommand{\para}[1]{
    \ifpreprint
        \paragraph{#1.}
    \else
        \subsection{#1}
    \fi
}

\ifpreprint
    \definecolor{lime}{HTML}{A6CE39}
    \DeclareRobustCommand{\orcidicon}{
    	\begin{tikzpicture} \draw[lime, fill=lime] (0,0) circle [radius=0.16] node[white] { {\fontfamily{qag}\selectfont \tiny ID} };
    	\draw[white, fill=white] (-0.0625,0.095) circle [radius=0.007];
    	\end{tikzpicture} \hspace{-2mm}
    }
    
    \foreach \x in {A, ..., Z}{\expandafter\xdef\csname orcid\x\endcsname{\noexpand\href{https://orcid.org/\csname orcidauthor\x\endcsname}{\noexpand\orcidicon} } }

\fi

\begin{document}

\ifpreprint
    \title{\theTitle}
    \author{\small{Stefan Clarke\orcidA{}$^1$, Gabriele Dragotto\orcidB{}$^1$, Jaime Fern\'{a}ndez Fisac\orcidC{}$^2$,} \\ 
    \small{and Bartolomeo Stellato\orcidD{}$^1$}\\[.5em]
    \small{\textit{$^1$ Department of Operations Research and Financial Engineering, Princeton University}}\\
    \small{\textit{$^2$ Department of Electrical and Computer Engineering, Princeton University}}}
    
\else
    \title{\theTitle}
    \author{Stefan Clarke$^1$, Gabriele Dragotto$^1$, Jaime Fern\'{a}ndez Fisac$^2$, and Bartolomeo Stellato$^1$\\[.5em]
    $^1$ Department of Operations Research and Financial Engineering, Princeton University, U.S.A.\\
    $^2$ Department of Electrical and Computer Engineering, Princeton University, U.S.A.}
\fi

\maketitle

\begin{abstract}%
\theAbstract
\end{abstract}

\section{Introduction}

Decision-making is rarely an individual task; on the contrary, it often involves several self-driven interacting agents. Non-cooperative game theory provides the ideal playground to model the outcome of multi-agent decision-making.
Among these outcomes, the most popular one for \emph{non-cooperative} and \emph{simultaneous} games is arguably the Nash equilibrium \citep{nash_equilibrium_1950}. Intuitively, whenever the game's agents play Nash equilibrium strategies, no agent has the incentive to deviate to another strategy, and the multi-agent system is stable. \citet{monderer_potential_1996} proved that Nash equilibria are guaranteed to exist in \emph{potential} games, a large subclass of non-cooperative and simultaneous games;
specifically, if the \emph{rationality parameters} of a potential game, \ie, the parameters describing the agents' payoffs and strategy sets, are common information and observable, there always exists a Nash equilibrium that is an extremum of a so-called \emph{potential function}. Several games with practical applications can be modeled as \emph{potential games}, for example, congestion games, Cournot games~\cite{monderer_potential_1996}, and certain urban driving games~\cite{zanardi_urban_2021}.

However, the agents' self-driven behavior frequently conflicts with the greater societal goals; in such cases, external regulators (\eg, governments) can intervene in the agents' interaction via incentives, laws, and regulations. For instance, the transportation company of a large city adds or removes subway rides starting from its understanding of the commuters' preferences. Therefore, accurately estimating the rationality parameters as a proxy of the agents' behavior is arguably the most critical ingredient for deploying effective regulatory interventions.
In the specific case of potential games, if the external regulator can directly observe the rationality parameters, any extrema of the potential function is a Nash equilibrium.
Assuming that the regulator can directly observe such rationality parameters is, however, rather unrealistic; in several contexts, these parameters are often unknown and not directly observable. 
Although the rationality parameters are seldom observable, outcomes such as Nash equilibria are often observable through data. The observability of outcomes, as opposed to the observability of the rationality parameters, is the starting assumption of the techniques we present here. In this paper, we propose an algorithm that learns the rationality parameters starting from some data containing the outcomes (\ie, Nash equilibria) of the past agents' interactions. Our algorithm enables external regulators to learn the agents' rationality parameters from data to deploy effective regulatory interventions. We assume that 
each agent solves a convex quadratic optimization problem parametrized in some unknown \emph{rationality} parameters and some \emph{context} parameters (\eg, some observable environmental conditions), and, further,
that there exists a convex quadratic potential function in the agents' variables.

The majority of the learning literature focused on the problem of learning the optimal strategy from the perspective of the agents \citep{daskalakis_near_2021,hartline_noregret_2015,kleinberg_multiplicative_2009,krichene_online_2015}. In this paper, we study the less explored topic of learning the rationality parameters from the perspective of an external regulator, as in \citep{fabiani_learning-based_2022,li_endtoend_2020,heaton_learn_2021,bertsimas_data-driven_2015,allen_using_2020}. In contrast with the previous literature, our algorithm:
\begin{enumerate*}[label=(\alph*.)]
\item learns rationality parameters that induce \emph{exact} and \emph{exact generalized} Nash equilibria, instead of approximate equilibria or Nash equilibria \citep{li_endtoend_2020,heaton_learn_2021,bertsimas_data-driven_2015}, and
\item is an \emph{anytime} algorithm, \ie, the intermediate estimates always induce exact Nash equilibria, as opposed to~\citep{fabiani_learning-based_2022,li_endtoend_2020,heaton_learn_2021,bertsimas_data-driven_2015,allen_using_2020}, and
\item learns both the agents' utilities and constraints, as opposed to only utilities~\citep{fabiani_learning-based_2022}, and
\item scales to high-dimensional datasets.
\end{enumerate*}
We summarize our contributions as follows:

\begin{enumerate}
\item We propose a stochastic \emph{active-set-based} first-order algorithm to learn the rationality parameters of the agents of a potential game given a dataset of past observed equilibria and context parameters. 
\item We provide a detailed analysis of the asymptotic convergence properties of our algorithm. 
\item We showcase the effectiveness and efficiency of our algorithm via computational experiments on Cournot and congestion games. 
\end{enumerate}
The code associated with this work is available at \url{https://github.com/stellatogrp/learning_rationality_in_potential_games}.

\section{Related work}
\label{sec:review}

\para{Learning in games} The majority of the learning literature focuses on the problem of learning optimal strategies from the perspective of the agents~\citep{daskalakis_near_2021,hartline_noregret_2015,kleinberg_multiplicative_2009,krichene_online_2015}. 
These approaches are useful to determine suitable agent strategies, yet it is hard to use these techniques to determine how to correctly \emph{intervene} in the game to achieve a specific purpose.
In this paper, we study the less explored topic of learning the parameters of the game, \ie, affecting the utilities and the strategies of the agents, from the perspective of an external observer~\citep{fabiani_learning-based_2022}. 
In contrast with~\citep{vorobeychik_learning_2007,ling_what_2018,waugh_computational_2011}, our approach 
does not employ normal or extensive-form formulations. Instead, we let each agent decide their strategy by solving a parametrized optimization problem, as in \citep{facchinei_generalized_2007, Dragotto_2022_Thesis}. Furthermore, compared to \citep{fabiani_learning-based_2022}, we allow agents to have individual constraints parametrized in both rationality and context parameters.

\para{Inverse equilibrium} Our work is closely related to the \emph{inverse equilibrium problem}, \ie, the problem using data to estimate the parameters affecting the game. 
Several works tackle these problems using variational inequalities~\citep{bertsimas_data-driven_2015} and implicit neural networks~\citep{heaton_learn_2021}, with applications in reinforcement learning~\cite{10.5555/645529.657801}, decision making~\cite{jibang_inverse_2022}, and motion planning~\citep{liu_learning_2022, GROVER2022167}.
However, these works often 
do not scale up to large datasets and do not provide anytime algorithms.
In contrast, our method models the inverse learning task as a (nonconvex quadratic) optimization problem, and, in this sense, it is closer to~\citep{agrawal_learning_2021} in the context of learning the parameters of convex optimization models. Similarly to \citep{jin2022learning}, we employ a complementarity problem for our learning task; however, our algorithm enforces exact complementarity at each iteration whereas~\cite{jin2022learning} incorporates it as a violation-based penalty function.
\vspace{-0.1em}
\new{
\para{Stochastic Methods}
Several authors employed stochastic gradient descent-type methods to estimate the parameters in games. Geiger and Straehle~\cite{geiger2022} predict rationality parameters in urban driving games using a neural network to anticipate the future trajectories of vehicles. They a neural network, the ``preference revelation net'', to map car trajectories to rationality parameters for each driver. The authors also optimize the neural network using the implicit function theorem. In particular, they train a neural network to predict the ``rationality'' parameters for a given class of trajectory games, specifically driving games. In contrast, in our work, we directly estimate thegame parameters without relying on the predictions from a neural network, and we consider quadratic potential games. 
Similarly, ~\cite{jibang_inverse_2022} model an inverse-equilibrium task in motion planning problems and focus on quantal response equilibrium, \ie, a different equilibrium concept. In contrast to the previous works, we model the learning task as an optimization problem with complementarity constraints and consider Nash equilibria.
Several works in bilevel optimization, \eg, \cite{zucchet2022beyond,arbel2022amortized}, combine stochastic methods and implicit differentiation to solve bilevel (\ie, nested) optimization problems; while these techniques do not focus on simultaneous games, they may be extended to estimate the rationality parameters in this setting.
}

\section{Problem formulation} 
\label{sec:formulation}
We consider a class of potential games among $\nagent$ \new{agents} solving optimization problems as in \cref{def:game}. We assume there are known and observable context parameters $\gameparam$ and some unknown rationality parameters $\param$, which we seek to infer. 
\begin{definition}[Mathematical programming game]
Let $\theta \in \reals^p$ be the vector of rationality parameters, and $\gameparam \in \reals^k$ be the vector of context parameters.
We consider a \emph{non-cooperative} and \emph{simultaneous} game among $\nagent$ agents such that each agent $i=1,\dots,\nagent$ solves the (parametrized) optimization problem
\begin{equation} \label{eq:formulation}
	\begin{array}[t]{ll}
		\underset{x_i \in \reals^m}{\mini} &  u_i(x_i;x_{-i}, \param, \gameparam) \\
		\text{\rm subject to} &   B_i(\theta, \gameparam) x_i + D_i(\theta, \gameparam) x_{-i} \leq b_i(\theta, \gameparam),
	\end{array}
\end{equation}
where 
$x \in \reals^{mn}$ is the strategy profile made of the decisions of all agents, including other agents' decisions $x_{-i}$ and the decisions~$x_i$ of~$i$; $B_i(\theta, \gameparam) \in \reals^{l \times m}$, $D_i(\theta, \gameparam) \in \reals^{l \times (n-1)m}$ and $b_i(\theta, \gameparam) \in \reals^{l}$; and
 $u_i(x_i;x_{-i}, \param, \gameparam)$ is a convex quadratic function in both $x_i$ and $x_{-i}$.
\label{def:game}
\end{definition}

Without loss of generality, each agent controls~$m$ variables, and, given the other agents' choices and parameters $\gameparam$, $\theta$, it optimizes a convex quadratic function subject to polyhedral constraints.
We say that $x^\star$ is a (generalized) Nash equilibrium of~\cref{eq:formulation} if it satisfies~\cref{def:Nash}.

\begin{definition}[Nash equilibrium]
A strategy profile $x^\star=(x^\star_1,\dots,x^\star_\nagent)$ is a Nash equilibrium for some given $\param$ and $\gameparam$ if, for any agent $i$, $u_i(x^\star_i;x^\star_{-i}, \param, \gameparam) \le u_i(x_i;x^\star_{-i}, \param, \gameparam)$ for any $x_i \in \{B_i(\param \new{\new{, \mu}}) x_i + D_i(\param \new{, \mu}) \new{x_{-i}^\star} \leq b_i(\param \new{, \mu}) \}$.
\label{def:Nash}
\end{definition}
We assume that the game is potential as in \cref{def:potential}.
\begin{definition}[Potential game \new{\cite{monderer_potential_1996}}]
\label{def:potential}
    A game of the form \ref{def:game} is a potential game if there exists a function $\Phi(x; \param, \gameparam)$ such that $x^\star$ is a Nash Equilibrium of the game if and only if $x^\star$ is the optimal solution of 
    \begin{equation} \label{eq:potential}
    \begin{array}{ll}
		\underset{x \in \reals^{mn}}{\mini} &  \Phi(x; \param, \gameparam) \\
		\text{\rm subject to}  &   B_i(\theta, \gameparam) x_i + D_i(\theta, \gameparam) x_{-i} \leq b_i(\theta, \gameparam) \\
        &  \hfill i=1,\dots,\nagent.
	\end{array}
    \end{equation}
\end{definition}
We assume that the potential function $\Phi(x; \param, \gameparam)$ is a convex quadratic function in $x$; \new{this implies, in most cases, that the agents have quadratic utility functions. Several games in economics~\cite{NBERw24914}, navigation~\cite{rosenthal_class_1973}, driving~\cite{zanardi_urban_2021} and robotics~\cite{robotics} involve agents with quadratic utility functions. } Let $R$ and $c$ be a matrix and a vector of appropriate dimensions, respectively. Then,
\begin{align*}
    \Phi(x; \param, \gameparam) &= (1/2)x^T R(\param, \gameparam) x + c(\param, \gameparam)^T x.
\end{align*}
Let the complementarity operator $\clubsuit \perp \spadesuit$ be equivalent to $ \clubsuit ^{\new{T}} \spadesuit=0$.
By taking the Karush–Kuhn–Tucker conditions of the problem of \cref{def:potential}, a vector $x^\star$ is a Nash equilibrium if and only if there exists a $\lambda^\star$ such that 
\begin{align*}
    & 0 = R(\param, \gameparam) x^\star + c(\param, \gameparam) + A(\param, \gameparam)^T \lambda^\star, \label{eq:learning_problem1:c1} \\
    & 0 \leq b(\param, \gameparam) - A(\param, \gameparam) x^\star \perp \lambda^\star \geq 0,
\end{align*}
where we define $A(\param, \gameparam)$ to be such that the inequality $0 \leq b(\param, \gameparam) - A(\param, \gameparam) x$ is equivalent to $B_i(\theta, \gameparam) x_i + D_i(\theta, \gameparam) x_{-i} \leq b_i(\theta, \gameparam)$ for every agent $i=1,\dots,\nagent$. 

\section{Our method}
\label{sec:method}
In this section, we describe the fundamental ingredients for our \new{active-set based} learning algorithm. \new{By \emph{active-set-based}, we mean that the algorithm identifies a subset of the active constraints and fixes them to $0$ (\ie, their bound) at each iteration to perform implicit differentiation of the loss function.}
Given a dataset $\mathcal{D}=\{(\bar{x}^k, \bar{\gameparam}^k)\}^K_{k=1}$ where the agents play Nash equilibria $\bar{x}^k$ for every $k=1,\dots,K$, our algorithm learns the parameters $\param$ such that each $(\bar{x}^k, \bar{\gameparam}^k)$ is as close to a Nash equilibrium as possible, \eg, so that the data $\mathcal{D}$ minimizes the prediction-based loss function with respect to the Nash equilibria induced by $\param$. For each datapoint $k \in \{1,\dots,K\}$, let $x^k= x^k(\param, \bar{\gameparam}^k)$ be the Nash equilibrium induced by some parameters $\param \in \Theta$, where $\Theta$ is the space of the rationality parameters. We will consider the loss function over $\mathcal{D}$ given by,
\begin{equation}
{\textstyle \loss(\param; \mathcal{D}) = (1 / |\mathcal{D}|) \sum_{\bar{x}^k \in \mathcal{D}} \parallel x^k - \bar{x}^k \parallel^2_2.}
\label{eq:loss_preliminary}
\end{equation}
To compute the points $x^k$, we devise a bilevel problem enforcing that $x^k$ is a Nash equilibrium given $\param$.

\subsection{Bilevel formulation}
\label{sub:bilevel}
For the given data $\mathcal{D}$ and \new{a closed, convex and known} set of possible parameters $\Theta$, we aim to find the parameters $\param \in \Theta$ that minimize, for each $k =1,\dots,K$, the squared $L_2$-norm of $x^k - \bar{x}^k$. Equivalently, we aim to find the $\param$ that solves the optimization problem
\begin{subequations}
    \label{eq:learning_problem1}
    \begin{align}
    		\!\!\!\! \underset{x^k, \lambda^k, \param}{\mini} \!\!\! \quad & \textstyle (1/K) \sum_{k=1}^K \| x^k - \bar{x}^k \|^2_2 \label{eq:learning_problem1:obj} \\
    		\!\!\!\!\! \text{subject to} \!\!\! \quad & 0 = R(\param, \bar{\gameparam}^k) x^k + c(\param, \bar{\gameparam}^k) + A(\param, \bar{\gameparam}^k)^T \lambda^k, \nonumber \\
        & 0 \leq b(\param, \bar{\gameparam}^k) - A(\param, \bar{\gameparam}^k) x^k \perp \lambda^k \geq 0   \nonumber  \\ & x^k \in \reals^{m n}, \lambda^k \in \reals^{l n}_+\quad  k = 1, \dots, K, \label{eq:learning_problem1:c1}\\
        & \param \in \Theta.
    \end{align}
\end{subequations}
The variables $x^k \in \reals^{n m}$ and $\lambda^k \in \reals^{l n}_+$ are the strategy profile and the dual variables associated with the agents' optimization problem, respectively, while $\theta$ are the variables associated with the rationality parameters. In \cref{eq:learning_problem1:c1}, we enforce that the strategy profiles $x$ are Nash equilibria of the game induced by $\param$ and the game parameters $\bar{\gameparam}$. In \cref{eq:learning_problem1:obj}, we minimize the average over data points of the squared $L_2$-distance between each data point $\bar{x}^k$ and $x^k$.
The optimization problem \cref{eq:learning_problem1} is nonconvex due to the complementarity constraints. We can obtain an equivalent formulation by explicitly setting either the LHS or the RHS of the complementarity constraints \cref{eq:learning_problem1:c1} to $0$.
For each data point $k$, we define an \emph{active-set} $Z^k \subseteq \{1,\dots,lm\}$, \ie, the set of indices of tight complementarity constraints; whenever $z \in Z^k$, the $z$-th constraint $b(\param, \bar{\gameparam}^k)_{z} - A(\param, \bar{\gameparam}^k)_{z} x^k = 0$ (is tight). Otherwise, if  $z \notin Z^k$, then $\lambda^k_z=0$.
For a given active-set $Z^k$, let $Y^k$ be its complement. Then, problem \cref{eq:learning_problem1} is equivalent to the following minimization problem over~$Z^k, x^k, \lambda^k, \param$,
\begin{subequations}
\label{eq:learning_problem2}
\begin{align}
	 \!\!\!\!\!\!\!\!\!\!\text{minimize} \quad & \textstyle (1/K) \sum_{k=1}^K \parallel x^k - \bar{x}^k \parallel^2_2 \\
	 \!\!\!\!\text{subject\! to} \quad & 0 = R(\param, \bar{\gameparam}^k) x^k + c(\param, \bar{\gameparam}^k) + A(\param, \bar{\gameparam}^k)^T \lambda^k, \label{eq:learning_problem2:c1}\\
    & 0 \leq b(\param, \bar{\gameparam}^k) - A(\param, \bar{\gameparam}^k)x^k, \quad \lambda^k \geq 0, \label{eq:learning_problem2:c2}\\
    & 0 = b(\param, \bar{\gameparam}^k)_{Z^k} - A(\param, \bar{\gameparam}^k)_{Z^k} x^k, \; 0= \lambda^k_{Y^k},\label{eq:learning_problem2:c3} \\ 
     & x^k \in \reals^{m n}, \lambda^k \in \reals^{l n}_+,\quad k = 1, \dots, K, \nonumber \\
    & \param \in \Theta \nonumber,
\end{align}
\end{subequations}
where we use the subscript $Z^k$ to refer to the set of constraints with indices contained in $Z^k \subseteq \{1,\dots,lm\}$, and $Y^k = \{1, \dots lm \} \backslash Z^k$.
In \cref{eq:learning_problem2:c2}, we relax the complementarity constraints of \cref{eq:learning_problem1:c1} by enforcing non-negativity on both the LHS and RHS involved in the complementarity constraints; in \cref{eq:learning_problem2:c3}, we enforce the tightness conditions on the active-sets $Z^k$ for any $k$. Although \cref{eq:learning_problem2} does not include any complementarity constraints, it involves exponentially many choices of $Z^k$ for any $k$. However, given some fixed $Z=(Z^1,\dots,Z^k)$, we can employ \cref{eq:learning_problem2} to formulate the loss function \cref{eq:loss_preliminary} and devise a differentiation method to iteratively refine the estimates of~$\param$.
\para{Implicit differentiation of the loss function} 
\label{sub:implicit}
We exploit the gradient information of the loss function \cref{eq:loss_preliminary} associated with problem \cref{eq:learning_problem2}. Let $x^k=x^k_{Z^k}(\param,\bar{\gameparam}^k)$ (we omit $Z^k$ when the context is clear) be implicitly defined as
\begin{align} 
    \!\!\!\!\!\! F(x^k, \lambda^k\new{; \param}) &=
    \begin{bmatrix} 
    R(\param, \bar{\gameparam}^k) x^k + c(\param, \bar{\gameparam}^k) + A(\param, \bar{\gameparam}^k)^T \lambda^k \\
    b(\param, \bar{\gameparam}^k)_{Z^k} - A_{Z} x^k \\
    \lambda_{Y^k}
    \end{bmatrix} \nonumber \\ &=0 \label{eq:implicitrelation},
\end{align}
\ie, $x^k$ solves $0= F(x^k, \lambda^k\new{; \param})$, or, in other words, $x^k$ is a Nash equilibrium given $Z$.
Therefore, for a given active-set $Z$, we reformulate the loss function \cref{eq:learning_problem2} over $\mathcal{D}$ as
\begin{equation}
{\textstyle \loss(\param, Z; \mathcal{D}) = (1/|\mathcal{D}|) \sum_{\bar{x}^k \in \mathcal{D}} \parallel x_{Z^k}^{k} - \bar{x}^k \parallel^2_2.}
\label{eq:loss}
\end{equation}
We evaluate the gradient of $\loss(\param, Z; \mathcal{D})$ with respect to $\param$ as
\begin{align}
\label{eq:grad}
    {\textstyle \!\!\!\!\!\!\nabla_\param \loss(\param, Z;\mathcal{D}) =\!(1/|\mathcal{D}|) \sum_{\bar{x}^k \in \mathcal{D}} 2\nabla_\param (x_{Z^k}^{k}) ^T( x_{Z^k}^{k} - \bar{x}^k).}
\end{align}
Equivalently, we can compute \cref{eq:grad} by employing the implicit function theorem \citep[Theorem 1B.1]{dontchev_2009}, as in \cref{thm:implicit}.

\begin{theorem}[Implicit function theorem] \label{thm:implicit}
Let $G(\new{\beta})$ be implicitly defined by the relation $0 = H(\new{\alpha, \beta})$ for $H: \reals^{d_1} \times \reals^{d_2} \mapsto \reals^{d_1} $. Assume 
\begin{enumerate*}
\item $H$ is continuously differentiable in a neighborhood of $\new{(\alpha_0, \beta_0)}$, 
\item $H(\new{\alpha_0, \beta_0}) = 0$, and
\item $\nabla_x H(\new{\alpha_0, \beta_0})$ is nonsingular.
\end{enumerate*} Then, $G$ is continuously differentiable in a neighborhood of $(\new{\alpha_0, \beta_0})$ with Jacobian,
\begin{equation*}
    \nabla_{\new{\beta}} G(\new{\beta}) = \nabla_{\new{\alpha}} H(\new{\alpha_0, \beta_0})^{-1} \nabla_{\new{\beta}} H(\new{\alpha_0, \beta_0}).
\end{equation*}
\end{theorem}
\new{We want to differentiate $x^k$ and $\lambda_k$ with respect to $\param$, where $x^k$ and $\lambda^k$ are defined implicitly by the relation \cref{eq:implicitrelation}. We, therefore, let $\alpha = (x^k, \lambda^k)$, $\beta = \param$ and $H(\alpha, \beta) = H((x^k, \lambda^k), \param) = F(x^k, \lambda^k; \param)$ so that $H(\alpha, \beta)$ is equal to}

\begin{equation} \label{eq:F}
    \begin{bmatrix}
    R(\param, \bar{\gameparam}^k) & A(\param, \bar{\gameparam}^k)^T \\ -A(\param, \bar{\gameparam}^k)_{Z^k} & 0 \\ 0 & I_{Y^k} 
    \end{bmatrix} \begin{bmatrix}
    x_k \\ \lambda_k
    \end{bmatrix} - 
    \begin{bmatrix}
    c(\param, \bar{\gameparam}^k) \\ b_{Z^k}(\param, \bar{\gameparam}^k) \\ 0
    \end{bmatrix}.
\end{equation}

\new{Using \cref{eq:F}, \cref{thm:implicit} allows us to differentiate $(x^k, \lambda^k)$ with respect to $\theta$, and, therefore,} to differentiate the loss function \cref{eq:loss} with respect to $\param$ for any active-set $Z$ of \cref{eq:learning_problem2}.

\subsection{The algorithm} 
\label{sub:algorithm}
We employ the ingredients of \cref{sub:bilevel,sub:implicit} to create a stochastic first-order method to learn the parameters $\param$ given the dataset $\mathcal{D}$. We formalize our approach in \cref{algo:activeset}. 
Given an instance of a game of \cref{def:game}, the dataset $\mathcal{D}$, a maximum number of iterations $T$, and a series of learning rates $\{\eta_t\}_{t=1}^T$, the algorithm returns the estimated rationality parameters $\param_{T}$.

    \begin{algorithm}
        \caption{Active-set algorithm} \label{algo:activeset}
        \begin{algorithmic}[1]
        
            \REQUIRE The data $\mathcal{D}=\{(\bar{x}^k, \bar{\gameparam}^k)\}^K_{k=1}$, the maximum number of iterations $T$, and step sizes $\{\eta_t\}_{t=1}^T$

            \STATE Initialize $\theta_0$  \label{algo:activeset:gaussian}

            \FOR{$t = 0, \dots, T-1$ \label{algo:activeset:for}}
            
            \STATE Choose $(\bar{x}^k, \bar{\gameparam}^k)$ uniformly from $\mathcal{D}$ \label{algo:activeset:pick}

            \STATE $x^{(t)}, \lambda^{(t)} \gets$ primal and dual variables from \cref{eq:potential} \label{algo:activeset:potential}

            \STATE $Z^{(t)} \gets \{ z : b(\param^{(t)}, \bar{\gameparam}^k)_z - A(\param^{(t)}, \bar{\gameparam}^k)_z x^{(t)} = 0 \}$ \label{algo:activeset:a2}

            \STATE $Y^{(t)} \gets \{z : \lambda^{(t)}_z = 0\}$  \label{algo:activeset:a3}

            \STATE Set $\loss_t(x) \gets \loss(\param^{(t)}, Z^{(t)};\{ \bar{x}^k\})$ \label{algo:activeset:loss}

            \STATE $\nabla_\param \loss_t(x^{(t)}) = \nabla_\param (x^{(t)}(\param^{(t)}))^T \nabla_x \loss_t(x^{(t)}(\param^{(t)}))$  \label{algo:activeset:gradient}

            \STATE $\param^{(t+1)} \gets \param^{(t)} - \eta_t \nabla_\param \loss_t(x^{(t)})$ \label{algo:activeset:update}
            
            \ENDFOR
            
            \RETURN $\param^{(T)}$
        \end{algorithmic}
    \end{algorithm}

After initializing $\theta^{(0)}$ (line \ref{algo:activeset:gaussian}), at each step $t$, the algorithm picks a random datapoint $(\bar{x}^k, \bar{\gameparam}^k)$ from $\mathcal{D}$ with uniform probability (line \ref{algo:activeset:pick}). Given $(\bar{x}^k, \bar{\gameparam}^k)$, the algorithm solves the optimization problem in line \ref{eq:potential} to compute $x^{(t)}$ and $\lambda^{(t)}$, with $x^{(t)}$ being a Nash equilibrium of the game induced by $\param^{(t)}$ and $\bar{\gameparam}^k$ (line \ref{algo:activeset:potential}). Let $Z^{(t)} = \{z :  b(\param^{(t)}, \bar{\gameparam}^k)_z - A(\param^{(t)}, \bar{\gameparam}^k)_z x^{(t)}= 0\}$ be the active-set at $x^{(t)}$, and $Y^{(t)} = \{z :  \lambda_z^{(t)}= 0\}$  (lines \ref{algo:activeset:a2} and \ref{algo:activeset:a3}). 
On the one hand, if $Z^{(t)} \cap Y^{(t)} = \emptyset$, \ie, there is no degeneracy in the complementarities, then the matrix in Step \ref{eq:F} is invertible. We can obtain the gradient of $x^{(t)}$ with respect to $\param^{(t)}$ using \cref{thm:implicit} (line \ref{algo:activeset:gradient}). We use this gradient information to update the estimates for $\param^{(t)}$ to reduce the loss proportionally to $\eta_t$, and we continue to iterate (line \ref{algo:activeset:update}).
On the other hand, if $Z^{(t)} \cap Y^{(t)} \neq \emptyset$, we have a degenerate solution. In the following subsection, we explain how to retrieve gradient information in degenerate active-sets. When the algorithm reaches its iteration limit, it returns the last estimate $\param^{(T)}$.

\para{Degenerate active-set}
Whenever, at some step $t$, $Z^{(t)} \cap Y^{(t)} \neq \emptyset$, we cannot obtain the gradient from \cref{eq:F} since the matrix is not invertible. This implies we must choose an alternative non-degenerate active-set on which we evaluate the gradient \cref{eq:grad}. We do so using the following rule, which is motivated by the theory in \cref{sec:convergence}.

\begin{zrule}
\label{rule:1}
    Suppose we are at iteration $t$ of \cref{algo:activeset}.
    \new{Update $\param^{(t+1)}$ using the gradient $\nabla_\param \mathcal{L}(\tilde{\param}; \{\bar{x}^k\})$, where $\tilde{\param}$ is chosen uniformly on  $\text{B}_\epsilon(\param^{(t)})$, \ie, the \new{Euclidean} ball of radius $\epsilon$.}
    \new{If it is the case that at $\tilde{\param}$,} $Z^{(t)} \cap Y^{(t)} \neq \emptyset$, let $W^{(t)} = Z^{(t)} \cap Y^{(t)}$. Randomly and uniformly partition $W(t)$ into two sets, $W_1^{(t)}$ and $W_2^{(t)}$. Let $\tilde{Z}^{(t)} = Z^{(t)} \backslash W_1^{(t)}$ and $\tilde{Y}^{(t)} = Y^{(t)} \backslash W_2^{(t)}$. Use $\mathcal{L}_t(x) \gets \mathcal{L}(\new{\tilde{\param}}, \tilde{Z}^{(t)}; \{ x_k \})$ in Step \ref{algo:activeset:loss} of \cref{algo:activeset}.
\end{zrule}

Essentially, \new{\cref{rule:1} updates the gradient at $\param^{(t)}$ using the gradient information from a $\tilde{\param}$, a slightly perturbed version of $\tilde{\param^{(t)}}$}. In the presence of degeneracy, \cref{rule:1} randomly partitions the intersection (\ie the elements causing the degeneracy) of $Z$ and $Y$ and uses this active-set to update the gradient in Step \ref{algo:activeset:gradient}. \new{However, we note that due to the perturbation, degeneracy rarely occurs in practice.}
\section{Convergence analysis}
\label{sec:convergence}
In this section, we will provide a proof of the convergence behavior of \cref{algo:activeset}. In particular, we will prove that a smoothed version of the gradient of the loss with respect to $\param^{(t)}$ at iteration $t$ converges to zero (\cref{thm:convergence}), \new{and, therefore, that our algorithm eventually finds either a local minimum of the smoothed loss or a saddle point. While this does not guarantee convergence to a global minimum, since we are providing an algorithm designed to run on a large dataset, we argue that a local minimum can empirically provide a good generalization of $\theta$, as we will show in our computational experiments.}
For notational simplicity, we will let $A(\param, \gameparam)$ be $A$.
First, for any $\param$ and any non-degenerate active-set $Z$ (\ie, $Z \cap Y = \emptyset$), \cref{eq:loss} provides an explicit expression for  $\nabla_\param \mathcal{L}( \param ; Z, \{ \bar{x}_k \})$ at any point $\bar{x}_k \in \mathcal{D}$.
Specifically, $\nabla_\param \mathcal{L} (\param ; Z, \mathcal{D} )$ is given by,
\begin{align*}
\nabla_\param \mathcal{L} (\param ; Z, \mathcal{D} ) &=
    \frac{2}{|\mathcal{D}|} \sum_{k=1}^K  S_k(Z_k, \param)^{-T} \begin{bmatrix}
        x^{k}_{Z_k} - \bar{x}^k \\ 0
    \end{bmatrix}, \\ S_k(Z_k, \param) &= \begin{bmatrix}
    R(\param, \bar{\gameparam}^k) & A^T \\ -A_{Z_k} & 0 \\ 0 & I_{Y_k} 
    \end{bmatrix},
\end{align*}
and the $j$-th row of the Hessian $\nabla^2 \mathcal{L}(\param, Z, \mathcal{D})_j$ is given by,
\begin{equation*}
    \frac{2}{|\mathcal{D}|} \sum_{k=1}^K S_k(Z_k, \param)^{-T}
    \partial_{\param_j} S_k(Z_k, \param) 
    S_k(Z_k, \param)^{-T}
    \begin{bmatrix}
        x^{k}_{Z_k} - \bar{x}^k \\ 0
    \end{bmatrix}.
\end{equation*}
To prove our result, we will assume the following.
\begin{assumption}[Constraint qualifications]\hphantom{.}
\label{ass:2}
\begin{enumerate}
    \item 
The linear independence constraint qualifications (LICQs) hold for the set $\{x \mid b(\param, \bar{\gameparam}^k) - Ax \geq  0\}$ for every $\param \in \Theta$, and $ k=1 \dots K $,
    \item 
the set $\{x \mid Ax \leq b(\param, \gameparam) \}$ is uniformly bounded in $x$ for all $\param$, $\gameparam$, and
    \new{\item for almost all $\param \in \Theta$ (in the Lebesgue-measure sense), $Z_k(\param) \cap Y_k(\param) = \emptyset$ for every $k=1 \dots K$.}
\end{enumerate}
\end{assumption}
\begin{assumption}[Well-behaved derivatives]\hphantom{.}
\label{ass:1} 
\begin{enumerate}
    \item there exists an $L_0 > 0$ such that, $ \|S_k(Z_k, \param)^{-T} \|_\infty < L_0$
    for every $\param \in \Theta, Z_k \subseteq \{1, 2, \dots, lm \}$ and $k = 1, \dots, K$, and
    \item the function $\partial_{\param_j} R(\param, \bar{\gameparam}^k)$ is bounded, namely, there exists an $L_3 > 0$ such that for all $k=1, \dots, K$ and all $\param \in \Theta$, $ \| \partial_{\param_j} R(\param, \bar{\gameparam}^k) \|_{\infty} \leq L_3$.
\end{enumerate}
\end{assumption}
\cref{ass:1,ass:2} are satisfied in several practical applications; for instance, in \cref{prop:assumptions}, we prove these assumptions hold for the games we consider in \cref{sec:experiments}.
\begin{proposition}
\label{prop:assumptions}
    Suppose that $b(\param, \gameparam)=b(\gameparam)$ has no dependency on $\param$, $A$ is constant, and $R$ and $c$ are affine functions of $\param$, \ie,
    \begin{align*}
        R(\param, \gameparam) &= R_0(\gameparam) + \textstyle \sum_{i=1}^p R_i(\gameparam) \param_i, \quad
        c(\param, \gameparam) = C(\gameparam) \param,
    \end{align*}
    with $R_0(\gameparam) \succ 0 $ and $R_i(\gameparam) \succeq 0$ for $i=1, \dots, p$.  If \cref{ass:2} holds, then \cref{ass:1} holds.
\end{proposition}
\begin{proof}
First, we prove that, under our assumptions, $S_k(Z_k, 0)$ is an invertible matrix. Otherwise, there would exist an $(x_0, \lambda_0)$ such that $S_k(Z_k, 0) (x_0, \lambda_0)^T = 0$ and therefore $x_0 = - R(0, \bar{\gameparam}^k)^{-1} A^T \lambda_0$, meaning that,
\begin{equation*}
    \begin{bmatrix}
        A_{Z_k} R(0, \bar{\gameparam}^k)^{-1} A^T \\
        I_{Y_k}
    \end{bmatrix} \lambda_0 = 0.
\end{equation*}
This is equivalent to,
\begin{equation*}
    {\lambda_0}_{Y_k} = 0, \quad A_{Z_k} R(0, \bar{\gameparam}^k)^{-1} A_{Z_k}^T{\lambda_{0Z_k}} = 0,
\end{equation*}
and since, by \cref{ass:2}, the matrix $A_{Z_k} R(0, \bar{\gameparam}^k)^{-1} A_{Z_k}^T$ is invertible. This implies that $\lambda_0 = 0$ and $x_0=0$, which is a contradiction.
Now we will prove that the infinity norm of $S_k(Z_k, \param)$ is bounded. Let $M_0$ and $M_1(\param)$ be defined as
\begin{equation*}
    M_0 = \begin{bmatrix}
        R_0(\gameparam) & A^T \\ -A_Z & 0 \\ 0 & I_Y
    \end{bmatrix}
     \quad M_1(\param) = \begin{bmatrix}
         \sum_{i=1}^p R_i(\gameparam) \param_i & 0 \\ 0 & 0 \\ 0 & 0
     \end{bmatrix},
\end{equation*}
so that $S_k(Z_k, \param) = M_0 + M_1(\param)$. Then
\begin{align}
    S_k(Z_k, \param)^T S_k(Z_k, \param) &= M_0^T M_0 + \textstyle \sum_{i=1}^p \param_i^2 R_i^T R_i \nonumber \\ & \quad + \textstyle \sum_{i=1}^p \param_i (R_i^T R_0 + R_0^T R_i).
    \label{eq:proof:1}
\end{align}
Since $R_i$ is positive semi-definite, $R_0$ is positive-definite, and $M_0$ is an invertible matrix, \cref{eq:proof:1} is positive-definite. Furthermore, its smallest eigenvalue is at least as large as the minimum eigenvalue $\epsilon_0$ of $M_0^T M_0$. Therefore, the largest eigenvalue of $(S_k(Z_k, \param)^T S_k(Z_k, \param))^{-1}$ is $1/\epsilon_0$ and so the greatest singular value of $S_k(Z_k, \param)^{-T}$ is $1/\sqrt{\epsilon_0}$. This means that the infinity norm of $S_k$ must be bounded. 
\end{proof}
We remark that $\nabla_\param \mathcal{L}(\tilde{\param}; \{\bar{x}^k\})$ is well-defined by \cref{ass:1} because $\tilde{\param}$ is a continuous random variable and the gradient is defined almost-everywhere.
For any fixed $\epsilon$ and $\param_0$, let $g$ be,
\begin{equation*}
     \textstyle g(\param) = (1/K) \sum_{k=1}^K \int_{\param_0}^\param \int_{\text{B}_\epsilon(v)} \frac{1}{\pi \epsilon^2} \nabla \mathcal{L}(u ; \{ \bar{x}^k \}) \text{d}u \text{d}v,
\end{equation*}
where the outer integral can be taken along any path from $\param_0$ to $\param$ since $\nabla \mathcal{L}(u; \{ \bar{x}^k \})$ is potential. $g$ can then be interpreted as a smoothed version of $\mathcal{L}$. We will show that \cref{algo:activeset} with the update rule \ref{rule:1} is equivalent to performing gradient descent on the smooth, nonconvex function $g$; this will allow us to provide our main convergence result.

\begin{theorem}
    For appropriately given step-sizes $\{\eta_t\}_{t=1}^T$ \new{given by $\eta_t = 1/t^{(3/4)}$,} the active-set algorithm with gradient Rule \ref{rule:1} satisfies $\lim_{T \rightarrow \infty}\Expect[\| \nabla g(\param_T) \|_2] = 0$.
    \label{thm:convergence}
\end{theorem}

\begin{proof}

We will begin by bounding relevant derivatives of $g$. Since terms in the relevant expressions for each derivative are bounded (using the boundedness of $S_k(S_k, \param)$ and 
$\partial_\param R(\param, \bar{\gameparam}^k)$), there exist some constants $L_1$ and $L_2 > 0$ such that, for all $k=1 \dots K$ and $Z \subseteq \{1, \dots, lm \}$,
\begin{equation*}
     \| \nabla \mathcal{L} (\param; Z, \{ X_k\}) \|_2 \leq L_1, \;\; -L_2 \preceq \nabla^2 \mathcal{L}(\param, Z, \{ X_k \}) \preceq L_2.
\end{equation*}
Note that $\Expect [\nabla\mathcal{L}(\tilde{\param})] = \nabla g(\param)$, where the randomness stems from the choice of the datapoint and the update \cref{rule:1}.
By the claim, $g$ is Lipschitz continuous with constant $L_1$ since,
\begin{align*} 
    \| \nabla g(\param) \|_2 &= {\textstyle \big\| (1/K) \sum_{k=1}^K \frac{1}{\pi \epsilon^2} \int_{\text{B}_\epsilon(\param)} \nabla \mathcal{L}(u ; \{ X_k \} ) \text{d}u \big\|_2} \\ 
    &\leq  {\textstyle (1/K) \textstyle \sum_{k=1}^K  \frac{1}{\pi \epsilon^2} \int_{\text{B}_\epsilon(\param)}\| \nabla \mathcal{L}(u ; \{ X_k \} ) \|_2 \text{d}u} \\ &\leq L_1.
\end{align*}
This and the Cauchy-Schwarz inequality also imply that,
\begin{equation*}
        \Expect [\| \nabla \tilde{\mathcal{L}}(\param) - \nabla g(\param) \|_2^2] \leq 4 L_1^2.
\end{equation*}
Since $\mathcal{L}$ is twice-differentiable on any $Z$, we have that
\begin{equation*}
    \textstyle \nabla^2 g(\param) = (1/K) \sum_{k=1}^K (1/(\pi \epsilon^2)) \int_{\text{B}_\epsilon(\param)} \nabla^2 \mathcal{L}(u ; \{ X_k \} ) \text{d}u. 
\end{equation*}
Since $\nabla^2 g$ is an average of $\nabla^2 \mathcal{L}$, we have that $-L_2 \preceq \nabla^2 g(\param) \preceq L_2$.
Using the facts above, the active-set algorithm with gradient \cref{rule:1} mirrors a stochastic gradient descent on $g$, a Lipschitz-continuous function with bounded second derivatives. We can then use a standard theorem about the convergence of stochastic gradient descent on smooth nonconvex functions such as \new{\cite[proposition 3]{convergence}}.
\end{proof}
Using the definition of $g$ again, we have shown that,
\begin{equation*} \textstyle
    \lim_{T \rightarrow \infty }\Expect_{\param^{(T)}} [ \| \Expect_\mathcal{U} [\nabla_\param \mathcal{L}(\param^{(T)} + \mathcal{U}; \mathcal{D})] \|_2 ] = 0,
\end{equation*}
where the outer expectation is taken over the steps of the algorithm, and the inner expectation is taken over the uniform distribution on the Euclidean ball $\mathcal{U} \sim U[B_\epsilon(0)]$.
\section{Computational results}
\label{sec:experiments}
In this section, we showcase the effectiveness of our algorithm via some computational experiments. We run the experiments on a \emph{Intel Xeon Gold 6246R} with $16$GB of RAM, and we use \emph{Gurobi 9.5} \cite{gurobi} as the optimization solver. For each experiment, we generate a training dataset $\mathcal{D}_{\text{train}}$ with $90$ samples, and a test dataset for evaluation $\mathcal{D}_{\text{test}}$ with $10$ datapoints. In both the datasets, the datapoints $\bar{x}_k$ are Nash equilibria given $\new{\bar{\param}}, \bar{\gameparam}_k$ with the addition of an independent Gaussian error vector $\epsilon_k \sim \mathcal{N}(0,\sigma)$ with $\sigma=0.001$. \new{We use a smoothing parameter $\epsilon=0.01$ for \cref{rule:1}.} We remark that the true rationality parameters $\new{\bar{\param}}$ are unknown. To assess the performance of \cref{algo:activeset} given the estimated parameters $\param^\star$, we employ the test error
\begin{equation*}
    \textstyle \sqrt{(1/| \mathcal{D}_{\text{test}}|) \sum_{(\bar{\mu}_k, \bar{x}_k) \in \mathcal{D}_{\text{test}}} \|x(\param^{\star}; \bar{\gameparam}_k) - \bar{x}_k \|_2^2}.
\end{equation*}

\subsection{Cournot Games} 
\label{p:cournot}
A generalized Cournot game is a simultaneous non-cooperative game where a series of agents decide the quantity of a homogeneous product to inject into a market. For every agent $i$, let the variable $x_i\ge 0$ represent the quantity of product produced, and the parameter $c_i \in \reals$ be the unit-cost of production. The price of the product is determined by the \emph{inverse demand function} $F(x) = a - b \sum_{j=1}^n x_j$, where $\param=(a,b)$ and $a,b \in \reals$. The utility (cost) function of $i$ is
\begin{equation*}
    { \textstyle u_i(x_i; x_{-i}, \param, \gameparam) = - F (x) x_i + c_i x_i} ,
\end{equation*} 
\ie, the revenues due to selling $x_i$ units at price $F$ minus the production cost $c_i x_i$. We define $\gameparam = (c_1,\dots,c_n)$ and the potential function $\Phi(x; \param, \gameparam)$ as
\begin{equation*}
    \textstyle c^T x + b x^2 + b \sum_{1 \leq i < j \leq n} x_i x_j - a^T x.
\end{equation*}
We consider $n=30, 50$, $70$, and $100$ agents, and we randomly generate $\bar{\param}_k$ and $\bar{\gameparam}_k$ according to the modulus of a Gussian distribution $\mathcal{N}(0,\sigma)$ with $\sigma=1$. We report the test error for the Cournot instances in the top half of \cref{fig:testerrors}. In all three cases, \cref{algo:activeset} learns accurate rationality parameters $\param$ inducing low test errors. Furthermore, in \cref{fig:cournot}, we compare the running times and the test errors of \cref{algo:activeset} with Gurobi solving the nonconvex problem \cref{eq:learning_problem1}. Since the solver runs out of time on our smallest instance with $n=50$, we generated some smaller instances with $n \in \{5,8,10,20,30,50\}$. We also reduced the size of the dataset, making it equal to the number of players, because Gurobi could not solve the problem with 100 data points in a reasonable time for any number of players.  While Gurobi times out (time limit of $2000$ seconds) whenever $n\geq10$, our algorithm solves all the instances. 

\begin{figure} 
    \vspace{3em}
    \centering
    \!\!\!\!\!\! 
    \ifpreprint
    \includegraphics[width=0.85\textwidth]{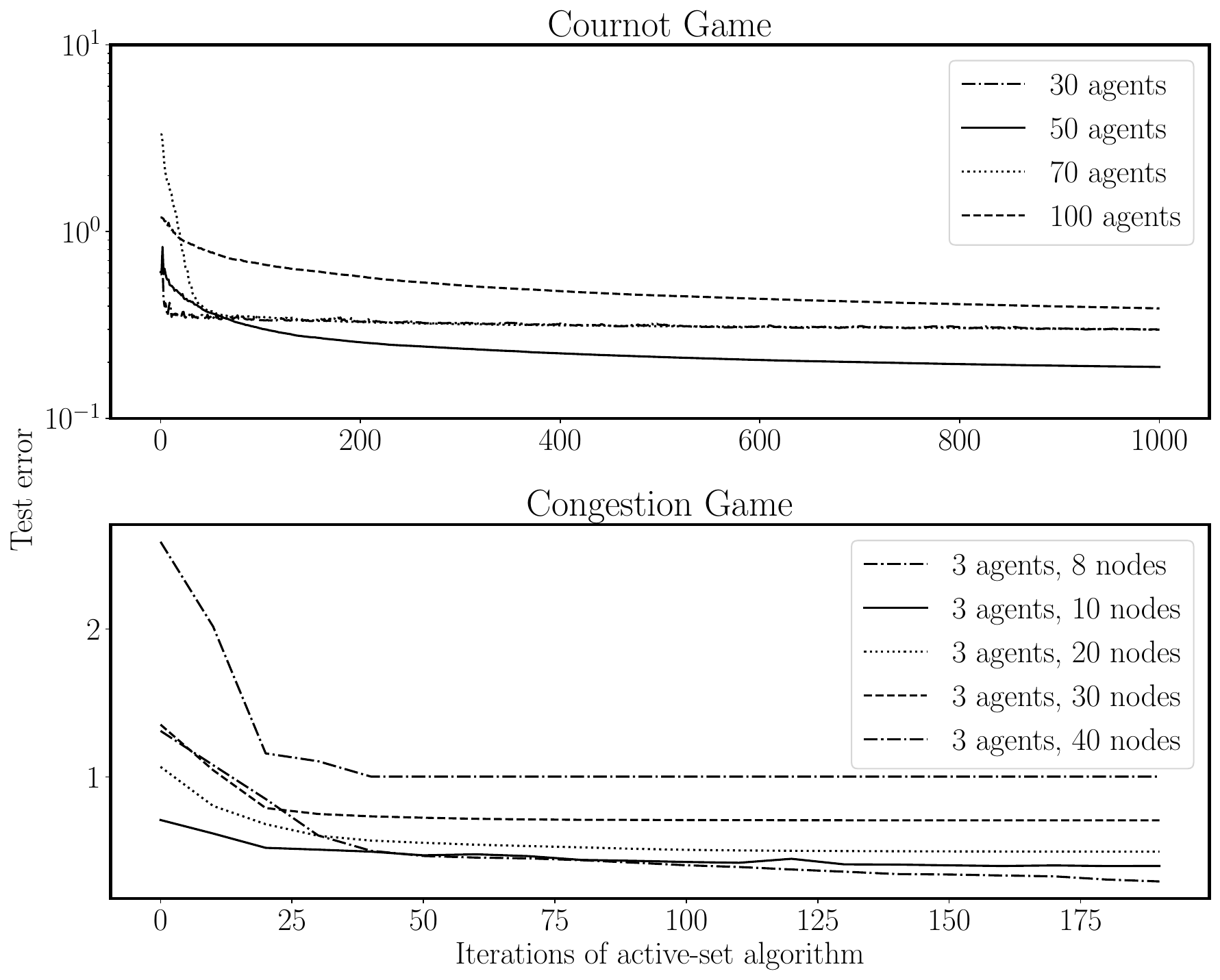}
    \else
    \includegraphics[width=0.45\textwidth,trim={0 1cm 0 2cm}]{data/lines.pdf}
    \fi 
    \caption{Test error for the Cournot and congestion games.} \label{fig:testerrors}
    \vspace{-1em}
\end{figure}

\begin{figure}
    \centering
    \ifpreprint
    \includegraphics[width=0.8\textwidth]{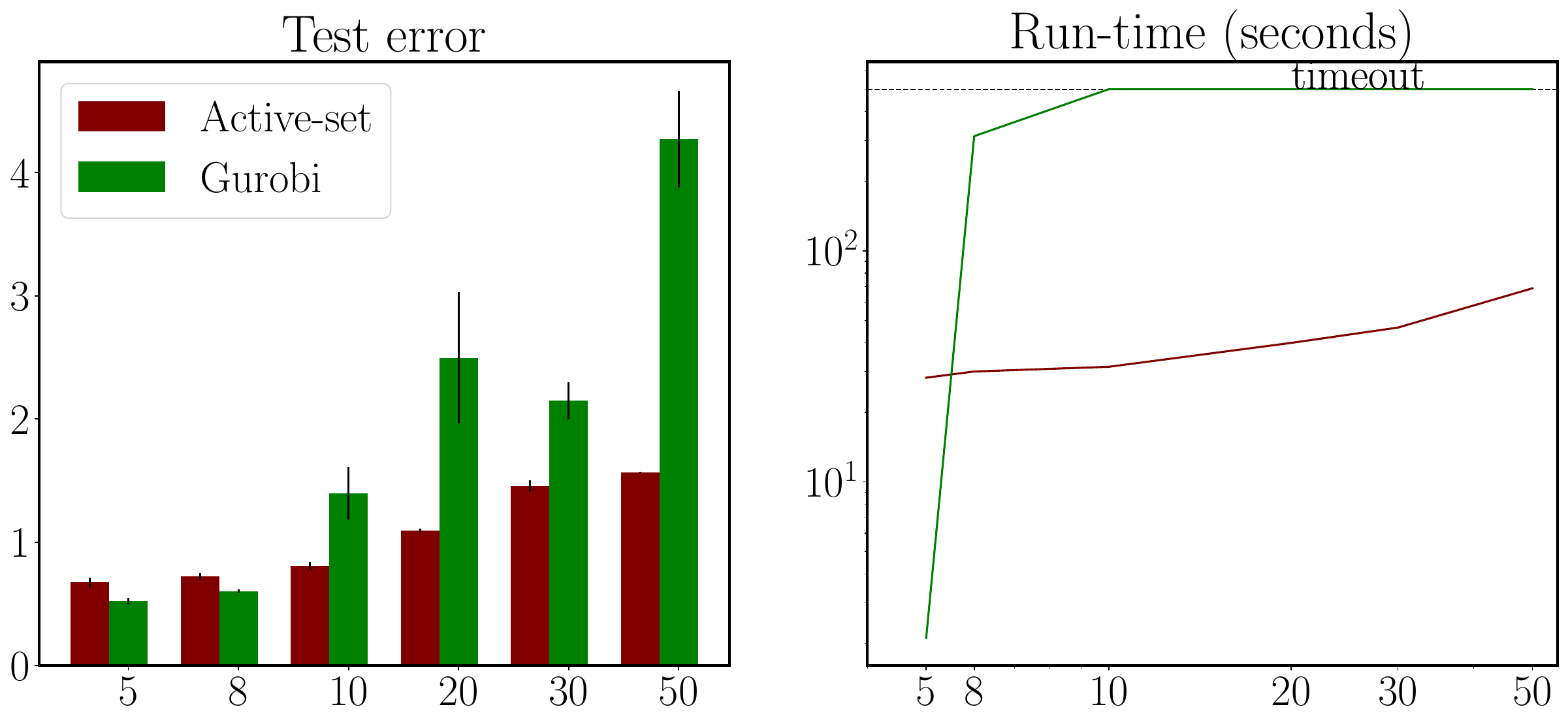}
    \else
    \includegraphics[width=0.40\textwidth,trim={0 2cm 0 0 }]{data/gurobi.pdf}
    \fi
    \caption{Comparison to Gurobi for the Cournot game.}
    \label{fig:cournot}
     \ifpreprint \else
    \vspace{-2em}
    \fi
\end{figure}

\subsection{Congestion games} 
We consider a class of congestion games \citep{rosenthal_class_1973} defined over a direct graph $G = (V, E)$, where $V$ is the set of vertices and $E$ is the set of edges. Each agent $i$ has to route $d_i \in \reals$ units of traffic from a node $s_i \in V$ to a destination $t_i \in V$. Let the variable $x_{ie}$ represent the amount of traffic agent $i$ routes through edge $e \in E$, and $C_e \in \reals$ be the cost parameter for edge $e$. 
Let $C=(C_{e_1}, \dots C_{e_{|E|}})$ be the vector of edge costs and $L\in \reals^{|E| \times p}$ be a matrix of factors relating to each edge (\eg, the length, width, and elevation of a road or the frequency of public transport). We assume that $C = L \param$ for some unknown rationality parameters $\param \in \reals^p$ and $\mu = {(s_i, t_i, d_i)}_{i \in I}$. All considered, each agent $i$ minimizes 
\begin{align*}
    u_i(x_i; x_{-i}, \param, \gameparam) &= \textstyle 
 \sum_{e \in E} C_e x_{ie} (x_{1e} + \dots + x_{ne}),
\end{align*}
subject to the standard flow constraints on $G$. The potential function of the game is
\begin{equation*}
    { \textstyle \Phi(x; \param, \gameparam) = \sum_{e \in E} C_e \sum_{i=1}^n ( (1/2) x_{ie}^2 + \sum_{j \neq i} x_{ie} x_{je} ) }.
\end{equation*}
We generate instances with $n=3$ and $|V| \in \{8,10,20,30,40\}$. The graphs $G$ are Erdős–Rényi random graphs where each edge has a probability $p=0.3$ of being present. Finally, we randomly generate $\bar{\param}_k$ according to the modulus of a Gussian distribution $\mathcal{N}(0,\sigma)$ with $\sigma=1$. In \cref{fig:testerrors}, we report the test errors on the congestion games. Similarly to the Cournot games, our algorithm converges to small test errors. In \cref{fig:toyexp}, we provide a small toy example on a congestion game with $n=2$.
\begin{figure}
    \vspace{0.5em}
    \centering
    \ifpreprint
    \includegraphics[width=0.8\textwidth, trim={5cm 2cm 0 2cm}]{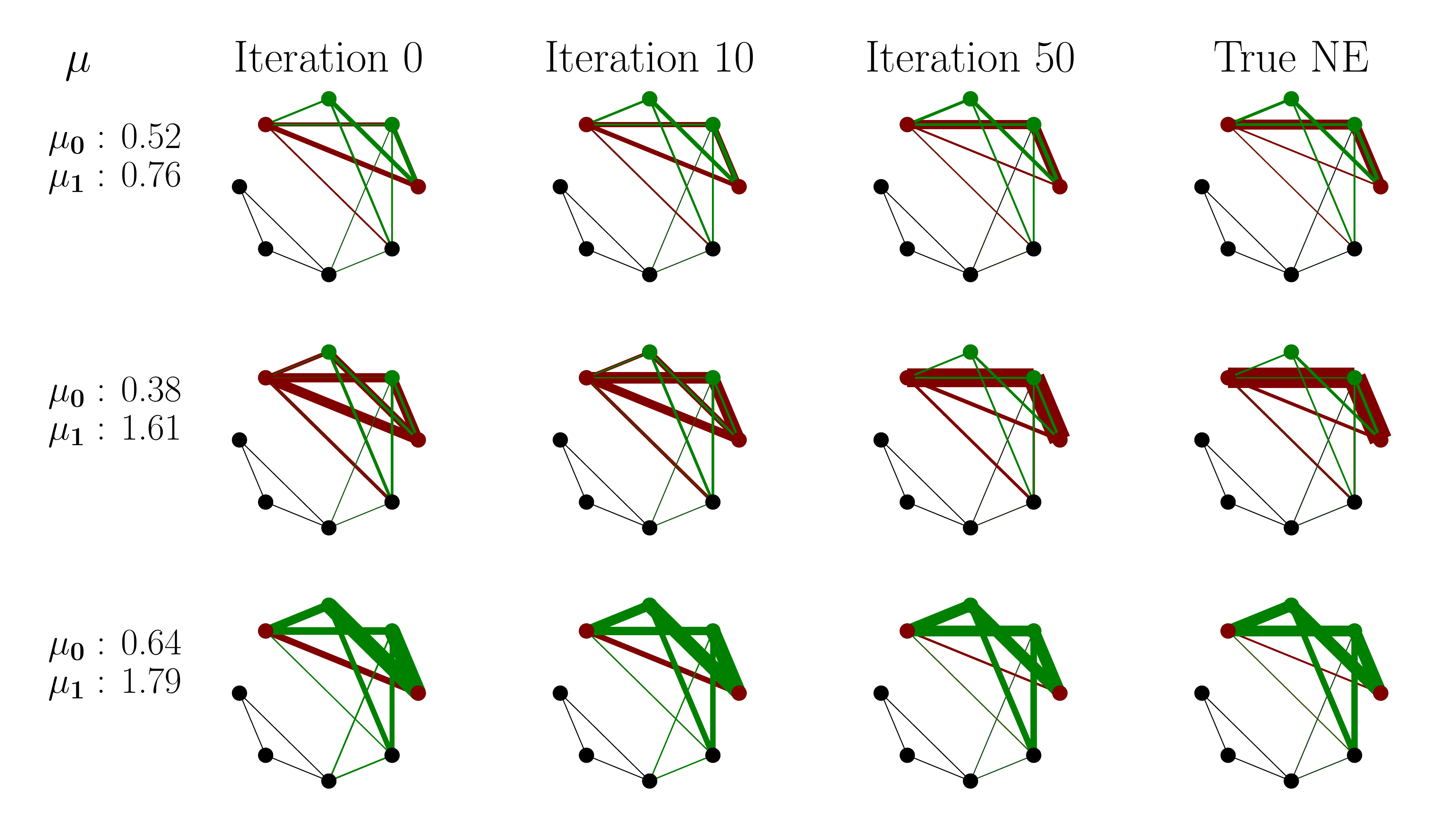}
    \else
    \includegraphics[width=0.45\textwidth, trim={5cm 2cm 0 2cm}]{data/graphs.pdf}
    \fi
    \caption{Three instances with different $\gameparam$ of a small congestion game with $n=2$ among the red and the green agent. The edge thickness represents the amount of traffic routed through the edge. \emph{True NE} represents the ground-truth Nash equilibrium.}
    \label{fig:toyexp}
    \vspace{-2em}
\end{figure}

\section{Conclusions}
We introduced a stochastic first-order algorithm to estimate the parameters of potential games. Our active-set-based algorithm alternates between an active-set step enforcing that the parameters induce Nash equilibria and an implicit-differentiation step to update the estimates of the parameters. From a theoretical perspective, we proved that our algorithm is guaranteed to converge. From a practical perspective, we demonstrated the efficiency of our technique via computational tests on Cournot and congestion games.

\ifpreprint
\newpage
\else
\balance
\fi

\bibliography{bibliography}

\end{document}